\newif\ifarxiv
\newif\ifpnas
\pnasfalse\arxivtrue

\ifarxiv
\documentclass[a4paper, 11pt]{article}
\usepackage[utf8]{inputenc}
\usepackage{amsmath, amsthm, amssymb}
\usepackage{xspace}
\usepackage{graphicx}
\usepackage{enumerate}
\usepackage[colorlinks=true, allcolors=blue]{hyperref}
\usepackage{xcolor}
\usepackage{booktabs}
\usepackage{cite}
\fi
\ifpnas
\documentclass[9pt,twocolumn, twoside]{pnas-new}
\templatetype{pnasresearcharticle} 
\usepackage{amsmath,amsthm}
\usepackage{xspace}
\usepackage{comment}
\fi
\usepackage{pifont}

\ifarxiv
\newcommand{\mysection}[1]{\section{#1}}
\newcommand{\mysubsection}[1]{\subsection{#1}}

\fi
\ifpnas
\newcommand{\mysection}[1]{\section*{#1}}
\newcommand{\mysubsection}[1]{\subsection*{#1}}

\fi

\newcommand{\Leskela}{Leskel\"a\xspace}

\definecolor{Green}{RGB}{0,170,0}
\definecolor{Red}{RGB}{230,0,0}

\ifarxiv
\newcommand{\eqsref}[2]{(\ref{#1}--\ref{#2})}
\fi
\ifpnas
\renewcommand{\eqref}[1]{[\ref{#1}]}
\newcommand{\eqsref}[2]{[\ref{#1}--\ref{#2}]}
\fi

\ifpnas
\setboolean{displaywatermark}{false}
\fi

\newlength{\mywidth}

\theoremstyle{definition}
\newtheorem{theorem}{Theorem}
\theoremstyle{plain}

\newtheorem{lemma}[theorem]{Lemma}
\newtheorem{proposition}[theorem]{Proposition}

\newcommand{\R}{\mathbb{R}}
\newcommand{\floor}[1]{\left\lfloor #1 \right\rfloor}

\newcommand{\abs}[1]{{\lvert#1\rvert}}
\newcommand{\norm}[1]{||#1||}
\newcommand{\zeronorm}[1]{||#1||_0}
\newcommand{\onenorm}[1]{||#1||_1}

\newcommand{\supnorm}[1]{||#1||_\infty}
\newcommand{\supnormt}[1]{\norm{#1}_{\infty,t}}
\newcommand{\Lipnormt}[1]{\norm{#1}_{{\rm Lip},t}}

\newcommand{\cmax}{c_\infty}

\newcommand{\weq}{\ = \ }

\newcommand{\wle}{\ \le \ }
\newcommand{\wge}{\ \ge \ }

\newcommand{\hu}{\hat u}

\newcommand{\halpha}{\hat \alpha}

\newcommand{\fbg}{\frac{\beta}{\gamma}}

\newcommand{\fgb}{\frac{\gamma}{\beta}}

\newcommand{\cX}{\mathcal{X}}
\newcommand{\Rnaught}{\mathcal{R}_0}

\title{Optimal intervention strategies for minimizing total incidence during an epidemic}

\ifarxiv
\author{Tom Britton\footnote{Department of Mathematics,
Stockholm University,
SE-106 91 Stockholm, Sweden}  \and Lasse \Leskela\footnote{Department of Mathematics and Systems Analysis,
Aalto University,
FI-02015 Espoo, Finland}}
\fi

\ifpnas
\author[a,1,2]{Tom Britton}
\author[b,1,2]{Lasse Leskel\"a} 
\affil[a]{
Department of Mathematics,
Stockholm University,
SE-106 91 Stockholm, Sweden
}
\affil[b]{Department of Mathematics and Systems Analysis,
Aalto University,
FI-02015 Espoo, Finland
}
\leadauthor{Britton}
\fi

\ifpnas
\significancestatement{During the Covid-19 pandemic the question of how much, and when, restrictions should be imposed, has received much attention in media and public health discussions. Here we formalize this discussion for a simplified epidemic model but allowing the intervention level to vary over time, subject to a budget constraint on the cumulative amount of intervention over time.
Using mathematical analysis and optimization theory we prove that 
the intervention strategy which minimizes the total number of infections and hospitalizations is 
a single constant-level lockdown of maximum possible magnitude.
The epidemic model is oversimplified, but the results still indicate that short severe restrictions might be more effective than longer periods of mild restrictions, and that imposing restrictions early is not efficient.
}
\authorcontributions{Author contributions:
TB and LL designed research, performed research, and wrote the article. LL wrote the code for numerical simulations.
}
\authordeclaration{The authors declare no conflict of interest.}
\equalauthors{\textsuperscript{1}TB~(Author One) contributed equally to this work with LL~(Author Two).}
\correspondingauthor{\textsuperscript{1}To whom correspondence should be addressed. E-mail: \url{tom.britton@math.su.se} or \url{lasse.leskela@aalto.fi}}
\fi

\ifarxiv
\begin{document}
\maketitle
\fi

\begin{abstract}
This article considers the minimization of the total number of infected individuals over the course of an epidemic in which the rate of infectious contacts can be reduced by time-dependent nonpharmaceutical interventions. The societal and economic costs of interventions are taken into account using a linear budget constraint which imposes a trade-off between short-term heavy interventions and long-term light interventions. We search for an optimal intervention strategy in an infinite-dimensional space of controls containing 
multiple consecutive lockdowns, gradually imposed and lifted restrictions, and various heuristic controls based for example on tracking the effective reproduction number. Mathematical analysis shows that among all such strategies, the global optimum is
achieved by a single constant-level lockdown of maximum possible magnitude. Numerical simulations highlight the need of careful timing of such interventions, and illustrate their benefits and disadvantages compared to strategies designed for minimizing peak prevalence. Rather counterintuitively, adding restrictions prior to the start of a well-planned intervention strategy may even increase the total incidence.
\end{abstract}

\ifpnas
\dates{This manuscript was compiled on \today}
\doi{\url{www.pnas.org/cgi/doi/10.1073/pnas.XXXXXXXXXX}}
\fi

\ifpnas
\begin{document}
\maketitle
\fi

\ifpnas
\thispagestyle{firststyle}
\ifthenelse{\boolean{shortarticle}}{\ifthenelse{\boolean{singlecolumn}}{\abscontentformatted}{\abscontent}}{}
\fi

\ifarxiv
%\tableofcontents
\newcommand{\startword}{The }
\section{Introduction}
\fi
\ifpnas
\newcommand{\startword}{\dropcap{T}he }
\fi

\startword recent pandemic has underlined the need for non-pharmaceutical interventions
to help mitigating disease burden in the society, along with vaccines and medications. Despite a solid body of past literature and an enormous research effort during the past two years on epidemic modelling, certain fundamental questions related to the optimal control of epidemics still remain open. In this article we discuss the optimal employment of non-pharmaceutical interventions to mitigate disease burden under the assumption that interventions incur societal costs which are accumulated over time. We focus on the minimization of long-term total incidence (the share of initially susceptible individuals who eventually become infected), and seek answers to questions of type:
\begin{quote}
\emph{Should interventions be imposed early, or later after prevalence has grown? Is it better to impose a one-month lockdown at 50\% intervention level, or
a milder two-month lockdown at 25\% intervention level?
}
\end{quote}

\noindent Numerical simulations and control theory are routinely used to answer such questions on a case-by-case basis for models of unlimited complexity. An alternative approach, pursued here, is to search for universal mathematical principles characterizing the shape and size of optimal interventions in simple parsimonious models.

The simplest mathematical epidemic model, incorporating a time-dependent intervention strategy is arguably defined as follows. We assume that the transmission rate of infectious contacts at time $t$ can be reduced by a factor
\[
 0 \le u(t) \le 1.
\] 
Under classical simplifying assumptions that recovery from disease gives full immunity, overall vaccination status in the population remains constant, there are no imported cases from other populations, population size remains constant, and the population is homogeneously mixing, the evolution of the epidemic can be modelled using differential equations
\begin{equation}
 \label{eq:Dynamics}
 \begin{aligned}
  S' &\weq - (1-u) \beta SI, \\
  I' &\weq (1-u) \beta SI - \gamma I,\\
  R' &\weq \gamma I,
 \end{aligned}
\end{equation}
where $\beta>0$ is a baseline transmission rate in absence of interventions, $\gamma > 0$ is the recovery rate, and $S(t)$, $I(t)$, and $R(t)$ represent the shares of susceptible, infectious, and recovered individuals in the population, respectively. The special case with no interventions ($u(t) = 0$ for all $t$) reduces to the constant-rate version of the classical SIR model~\cite{Kermack_McKendrick_1927}.

We write $(S,I,R) = (S_u,I_u,R_u)$ to emphasize that the epidemic trajectory depends on the chosen intervention strategy $u$. The societal and economic costs incurred by adopting an intervention strategy $u$ can by measured by 
\begin{itemize}
\item total cost $\onenorm{u} = \int_0^\infty u(t) \, dt$,
\item total duration $\zeronorm{u} = \int_0^\infty 1(u(t)>0) \, dt$,
\item maximum intervention level $\supnorm{u} = \sup_{t \ge 0} u(t)$.
\end{itemize}
For this model, the basic reproduction number equals $\Rnaught = \fbg$.  In what follows, we will assume that $\Rnaught > 1$ because otherwise epidemic outbreaks would not happen even in the absence of interventions \cite{Diekmann_Heesterbeek_Britton_2013}.

\mysubsection{Minimizing peak prevalence}
Recent theoretical research \cite{Morris_Rossine_Plotkin_Levin_2021,Greene_Sontag_2021,Miclo_Spiro_Weibull_2020,Avram_Freddi_Goreac_2022} on minimizing disease burden has mostly focused on minimizing the peak prevalence
\[
 \supnorm{I_u}
 \weq \sup_{t \ge 0} I_u(t).
\]
Morris et al.\ \cite{Morris_Rossine_Plotkin_Levin_2021} proved that the peak prevalence subject to interventions of bounded duration $\zeronorm{u} \le c_0$ is minimized by an intervention of form
\begin{equation}
 \label{eq:WaitMaintainSuppressRelax}
 u(t)
 \weq
 \left\{
 \begin{aligned}
  &0,
  \quad & t & \in (0,t_1]     & \quad & \text{(wait)} \\
  &1 - \tfrac{1}{(\beta/\gamma) S(t)},
  \quad & t & \in (t_1, t_2] & \quad & \text{(maintain)} \\
  &1,
  \quad & t & \in (t_2, t_3] & \quad & \text{(suppress)} \\
  &0,
  \quad &t & \in (t_3, \infty) & \quad & \text{(relax)},
 \end{aligned}
 \right.
\end{equation}
and discovered that such interventions induce a second wave having a peak of same height as the first wave.  Such a twin peaks phenomenon was also noted in~\cite{Greene_Sontag_2021}. Miclo, Spiro, and Weibull \cite{Miclo_Spiro_Weibull_2020} studied a dual problem of minimizing the intervention cost $\norm{u}_1$ subject to a bounded peak prevalence $\supnorm{I_u} \le \cmax$, and proved that the optimum is of form
\begin{equation}
 \label{eq:WaitMaintainRelax}
 u(t)
 \weq
 \left\{
 \begin{aligned}
  &0,
  \quad & t & \in (0,t_1]     & \quad & \text{(wait)} \\
  &1 - \tfrac{1}{(\beta/\gamma) S(t)},
  \quad & t & \in (t_1, t_2] & \quad & \text{(maintain)} \\
  &0,
  \quad &t & \in (t_2, \infty) & \quad & \text{(relax)}.
 \end{aligned}
 \right.
\end{equation}
The maintain phase in \eqsref{eq:WaitMaintainSuppressRelax}{eq:WaitMaintainRelax}
is defined so that $I'_u=0$, which keeps the infectious share at a constant level. A similar control problem restricted to a finite time horizon is analysed in~\cite{Avram_Freddi_Goreac_2022}.

\mysubsection{Minimizing total incidence}

Recent theoretical works on minimizing total incidence include  \cite{Feng_Iyer_Li_2021,Bliman_Duprez_Privat_Vauchelet_2021,Ketcheson_2021,Bliman_Duprez_2021}. Under different budget constraints on intervention costs,
they all conclude 
that the optimal interventions are constant-level lockdowns with shape
\begin{equation}
 \label{eq:WaitSuppressRelax}
 u(t)
 \weq
 \left\{
 \begin{aligned}
  &0,
  \quad & t & \in (0,t_1]     & \quad & \text{(wait)} \\
  &c,
  \quad & t & \in (t_1, t_2] & \quad & \text{(suppress)} \\
  &0,
  \quad &t & \in (t_2, \infty) & \quad & \text{(relax)}.
 \end{aligned}
 \right.
\end{equation}
Feng, Iyer, and Li \cite{Feng_Iyer_Li_2021} considered on-off controls with finitely many switching times, with duration $\zeronorm{u} = c_0$ and maximum level $\supnorm{u} = \cmax$.
In parallel works, Bliman et al.~\cite{Bliman_Duprez_Privat_Vauchelet_2021} and Ketcheson~\cite{Ketcheson_2021} studied piecewise continuous interventions subject to $\supnorm{u} \le \cmax$ and a bounded intervention time window $[0,T_0]$. Bliman and Duprez \cite{Bliman_Duprez_2021}
considered piecewise continuous interventions with bounded duration $\zeronorm{u} \le c_0$
and level $\supnorm{u} \le \cmax$,
and described a numerical
method to determine the optimal start time $t_1$. The analysis in \cite{Bliman_Duprez_2021} represents the state-of-art, covering the full class of piecewise continuous intervention strategies, and taking properly into account the accumulation of intervention costs over time.

The results above provide elegant mathematical principles describing optimal intervention shapes, but they all are limited in one important aspect. Namely, they ignore the fact that interventions of higher magnitude usually inflict a higher societal cost, and as such cannot help in answering questions related to the trade-off between the magnitude and duration of an intervention.

\mysection{Theoretical results}

\mysubsection{Optimal intervention strategy}

Our goal is to minimize total incidence among intervention strategies with total time-aggregated cost bounded by $||u||_1\le c_1$ and maximum intervention level bounded by $\supnorm{u} \le \cmax$ for some $0 < \cmax \le 1$. The latter constraint reflects the fact that a complete lockdown might be impossible to implement in practice.
We search for a global optimum in an infinite-dimensional space of piecewise continuous intervention strategies containing multiple consecutive lockdowns, gradually imposed and lifted restrictions, and various heuristic controls based for example on tracking the effective reproduction number \cite{Cianfanelli_etal_2022}. The following theorem shows that among all such intervention strategies, the global optimum is achieved by a simple constant-level lockdown.
 
\begin{theorem}
\label{the:Main}
For any initial state with $S(0), I(0)>0$, the total incidence among all piecewise continuous intervention strategies such that $\onenorm{u} \le c_1$ and $\supnorm{u} \le \cmax$ is minimized by an intervention of form \eqref{eq:WaitSuppressRelax} with
level $c=\cmax$,
duration $t_2-t_1 = c_1/\cmax$,
and a uniquely determined start time $t_1$.
\end{theorem}

Theorem~\ref{the:Main} provides a simple answer to the simple but mathematically nontrivial question presented in the beginning, indicating that heavy lockdowns of short duration outperform light lockdowns of longer duration.
The optimal start time $t_1$ may be numerically determined by solving a one-dimensional optimization problem, as  described in \cite{Bliman_Duprez_2021,Feng_Iyer_Li_2021}.

\mysubsection{Upper and lower bounds}

As a byproduct of the mathematical analysis needed for proving Theorem~\ref{the:Main}, we obtain universal upper and lower bounds for total incidence, valid for all intervention strategies with a finite cost. The upper bound corresponds to total incidence in the absence of interventions, and is expressed in terms of the limiting susceptible share $S_0(\infty)$ in a standard SIR epidemic with no interventions, which is numerically obtained as the unique solution in the interval $(0,\fgb)$ of
equation
\[
 S_0(\infty) - \fgb \log S_0(\infty)
 \weq S(0)+I(0) - \fgb \log S(0).
\]

\begin{theorem}
\label{the:Bounds}
For any initial state with $S(0),I(0) > 0$ and
for any intervention strategy with finite cost $\onenorm{u} < \infty$, the total incidence is at least
$1-\gamma /(\beta S(0))$ and at most
$1-S_0(\infty)/S(0)$
\end{theorem}

\mysection{Numerical results}

The performance of the optimal strategy in Theorem~\ref{the:Main} is investigated by numerical
\ifpnas
simulations
\fi
\ifarxiv
simulations\footnote{\url{https://github.com/lasseleskela/epidemic-models-with-control}.}
\fi
using parameters with basic reproduction number $\Rnaught = 3$ and an average infectious period of 5 days (Table~\ref{tab:Parameters}),
roughly in line with the first wave of the COVID-19 pandemic in spring 2020, e.g.\ \cite{Flaxman_etal_2020}. The initial state corresponds to an importation of 1000 infectious individuals into a population of size 10 million.

\begin{table}[h]
 \centering
 \begin{tabular}{lll}
  \toprule
  \bf Parameter & \bf Value & \bf Meaning \\
  \midrule
  $\beta$ & 0.6 & Transmission rate (per day) \\
  $\gamma$ & 0.2 & Recovery rate (per day) \\
  $S(0)$ & 0.9999 & Initial susceptible share \\
  $I(0)$ & 0.0001 & Initial infectious share \\
  \bottomrule
 \end{tabular}
 \caption{Parameters used in numerical simulations.}
 \label{tab:Parameters}
\end{table}

\mysubsection{Performance}

When an epidemic with a basic reproduction number $\Rnaught = 3$ hits an overwhelmingly susceptible population, eventually 94.0\% of individuals will become infected if no interventions are imposed. Furthermore, at least $1-1/(S(0)\Rnaught) = 66.6\%$ will become infected under an arbitrary intervention with a finite total cost (Theorem \ref{the:Bounds}).
The top panel in Fig.~\ref{fig:PerformanceLevel} displays the total incidence that is achievable using optimal strategies with total cost $\onenorm{u}$ bounded by $c_1 = 7.5, 15, 30$ and maximum intervention level bounded by $0 < \cmax \le 1$. It is seen that interventions with even a modest budget of $c_1=7.5$ may significantly reduce total incidence if they can be imposed at a sufficient magnitude.
On the other hand, long but mild interventions have little effect on total incidence. 

The bottom panel in Fig.~\ref{fig:PerformanceLevel} shows that the higher the maximal intervention level $\cmax$, the longer one should wait before imposing the intervention, and that mild interventions should be started immediately. For a complete lockdown ($\cmax=1$), the optimal timing is to wait for 24.94 days, corresponding to the time instant at which the uncontrolled epidemic reaches herd immunity.

\begin{figure}[h]
\centering
\setlength{\mywidth}{85mm}
\includegraphics[width=\mywidth]{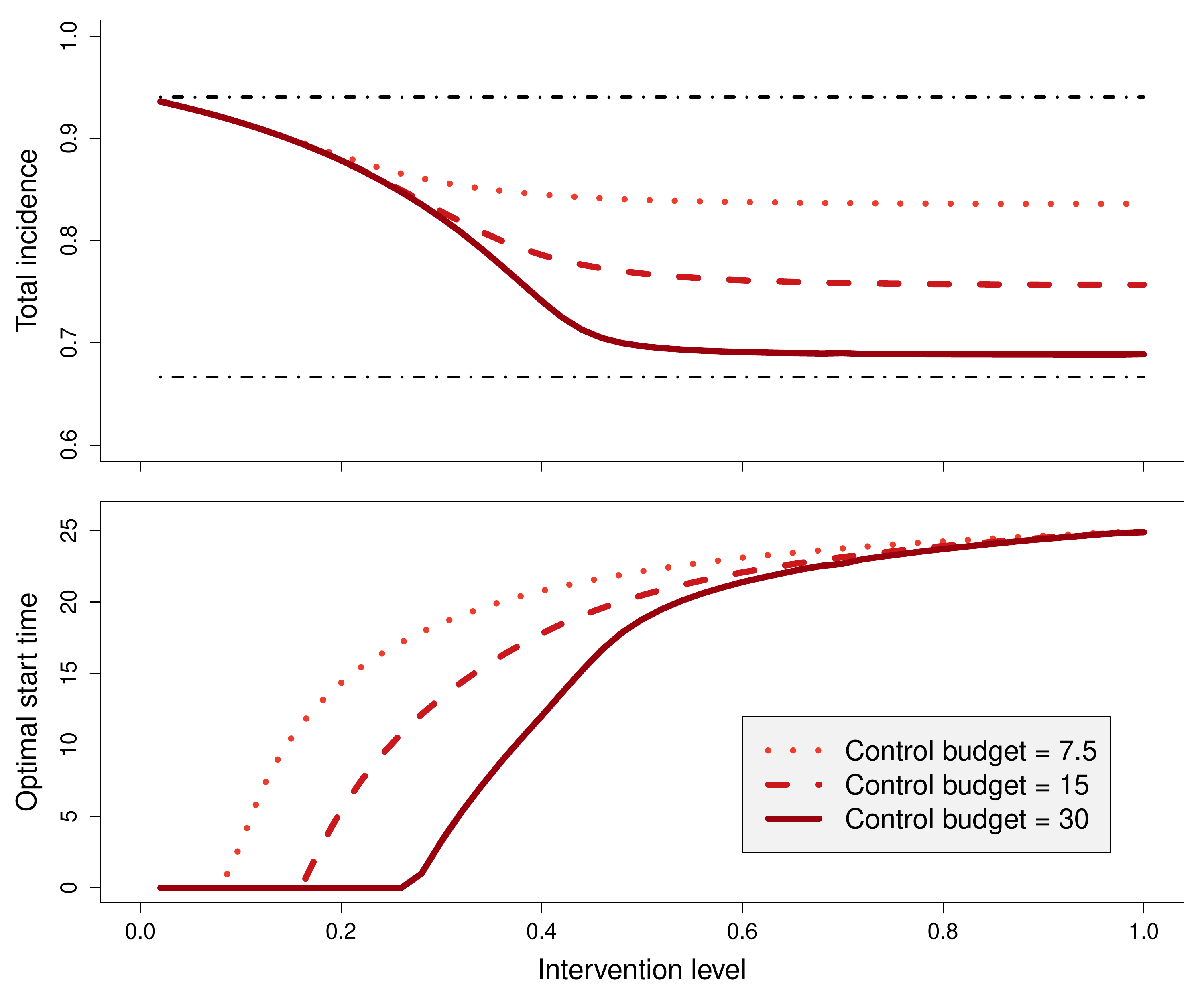}
\caption{Minimum total incidence (top panel) achievable using interventions with total cost $\onenorm{u} \le c_1$ and maximum level $\supnorm{u} \le \cmax$ for 
$c_1=7.5, 15, 30$, and a full range of $\cmax$ (horizontal axis).
The horizontal lines at levels 0.666 and 0.940 indicate the lower and upper bounds of Theorem~\ref{the:Bounds}.
The bottom panel displays the start times of the optimal strategies corresponding to the $(c_1,\cmax)$-pairs.
}
\label{fig:PerformanceLevel}
\end{figure}

\mysubsection{Optimal start time}

Fig.~\ref{fig:PerformanceStartTime} displays the total incidence 
of a 20-day constant-level lockdown having level 0.75 and total cost 15, for different values of the start time, and underlines the importance of proper timing of such interventions. Mistiming the lockdown even by just one week may have a big effect on number of eventually infected individuals. Somewhat strikingly, starting too early is about equally as bad as starting too late. Similar findings have been reported in \cite{Bliman_Duprez_2021,DiLauro_Kiss_Miller_2021}. Proper timing is crucial also when minimizing peak prevalence, as noted in \cite{Morris_Rossine_Plotkin_Levin_2021}.

\begin{figure}[h]
\centering
\setlength{\mywidth}{85mm}
\includegraphics[width=\mywidth]{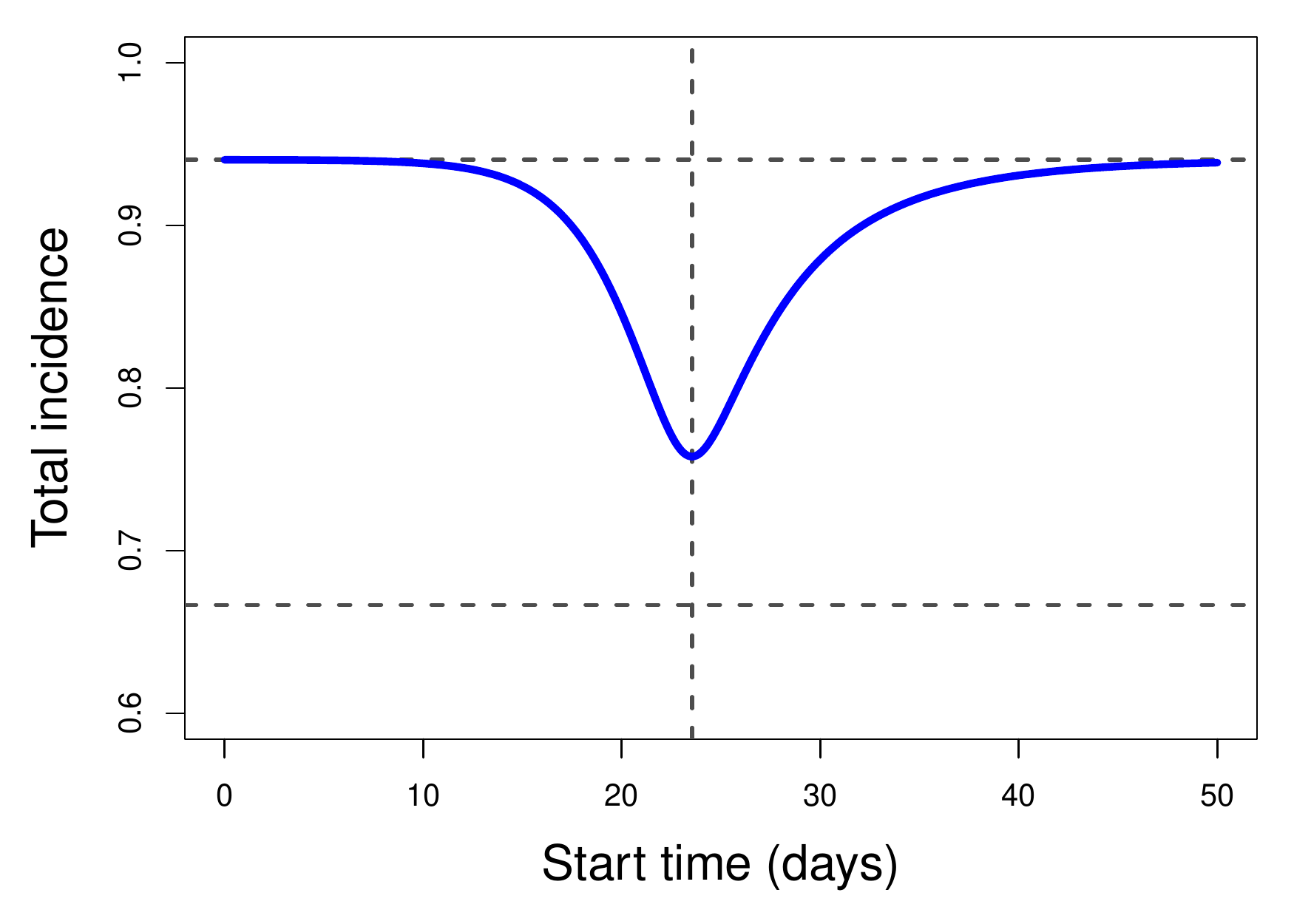}
\caption{
Total incidence of a constant-level lockdown of duration 20 days and level 0.75, for different values of the start time. The optimal start time $t_1=23.6$ yields a total incidence of 0.758. The horizontal lines at levels 0.666 and 0.940 indicate the lower and upper bounds of Theorem~\ref{the:Bounds}.
}
\label{fig:PerformanceStartTime}
\end{figure}

\mysubsection{Lack of monotonicity}

Fig.~\ref{fig:PerformanceStartTimeHighlighted} displays a striking phenomenon where adding restrictions before the start of an optimal intervention leads to more individuals eventually becoming infected. The original intervention (blue) with a duration 20 days and a constant level of 0.75, is started optimally at time 23.6. An alternative intervention (red) has the same level and end time but is started seven days earlier.
Such a modification, aimed at mitigating disease burden by adding more restrictions, actually has an opposite effect and leads to 19.7\% \emph{more} infections in the long run. This negative outcome is caused by a second wave of infections which starts when the alternative intervention is lifted. This again underlines the important of proper timing: if the longer longer 27-day lockdown were started optimally at time 23.4, the total incidence would have been 0.723 which is not much higher than the best possible outcome of $1 - 1/(S(0) \Rnaught) = 0.666$ given in Theorem~\ref{the:Bounds}. 
Similar observations manifesting the counterintuitive lack of monotonicity has been reported in \cite{Chikina_Pegden_2020,Handel_Longini_Antia_2007,Kruse_Strack_2020}.
On a positive note, extending interventions from the end will never do harm in this way (see Lemma \ref{the:Monotonicity} in the appendix).

\begin{figure}[h]
\centering
\setlength{\mywidth}{85mm}
\includegraphics[width=\mywidth]{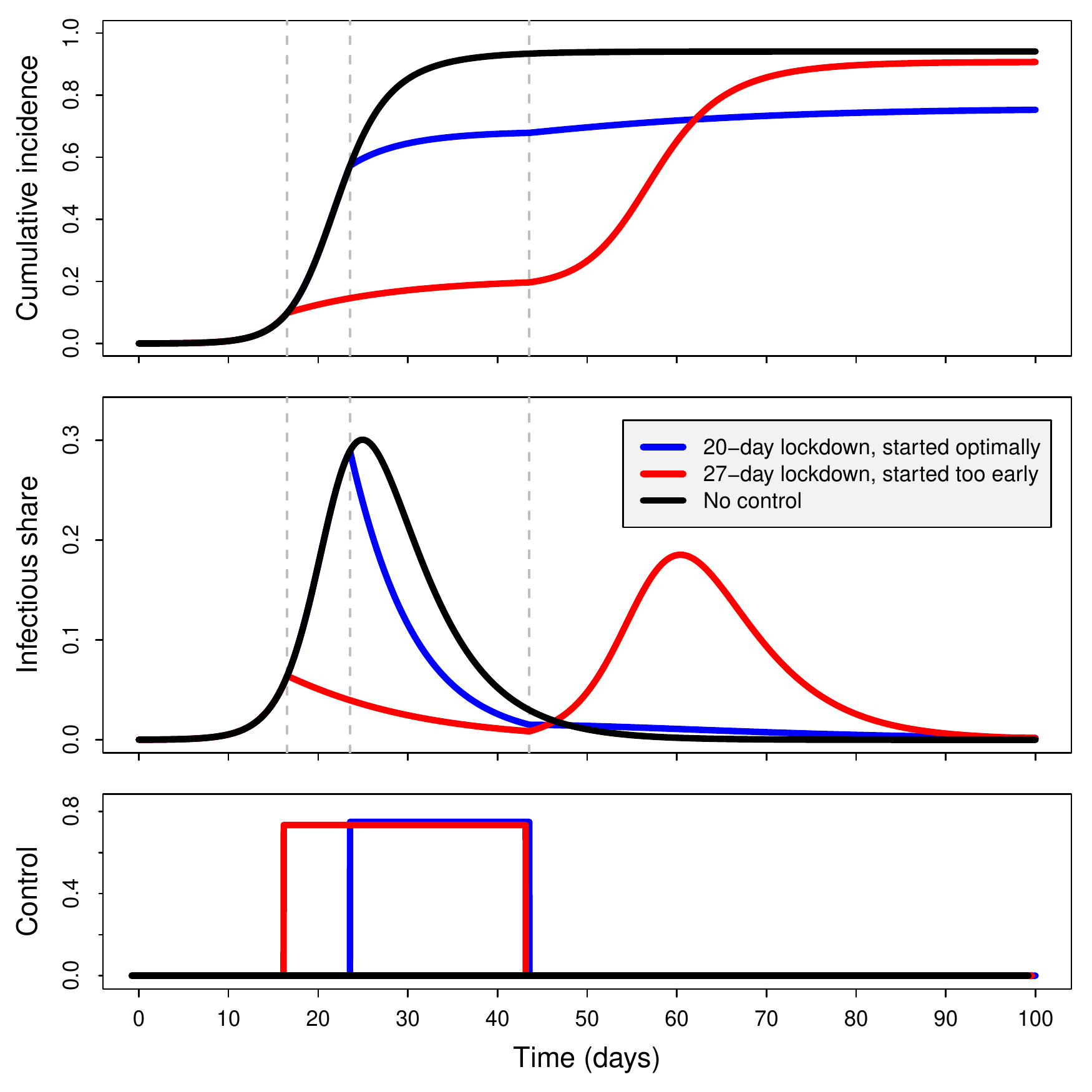}
\caption{Intervening too early may do more harm than good. An optimally timed 20-day lockdown at 0.75 level is in effect during time interval (23.6,43.6) and leads to a total incidence of 0.758 (blue). A longer but badly timed lockdown at the same level is effective during time interval (16.6, 43.6) and leads to a total incidence of 0.907 (red) which is not far from the total incidence of 0.940 corresponding to an uncontrolled epidemic (black). For the longer 27-day lockdown, the optimal start time is 23.4 and yields total incidence 0.723.
}
\label{fig:PerformanceStartTimeHighlighted}
\end{figure}

\mysubsection{Comparison with peak-minimizing strategy}

Fig.~\ref{fig:PerformanceVsMSW} displays time plots of a constant-level lockdown of type \eqref{eq:WaitSuppressRelax} designed to minimize total incidence and an intervention of type \eqref{eq:WaitMaintainRelax} designed to minimize peak prevalence. 
Both intervention strategies are constrained by a total budget $c_1=15$ and maximum intervention level $\cmax = 0.75$. The latter strategy imposes restrictions earlier, and when these are lifted, the share of infectious individuals in the population is still relatively high, leading to more additional infections and a higher total incidence. As expected, the former strategy yields a lower total incidence than the latter, but at the cost of higher peak prevalence.

\begin{figure}[h]
\centering
\setlength{\mywidth}{85mm}
\includegraphics[width=\mywidth]{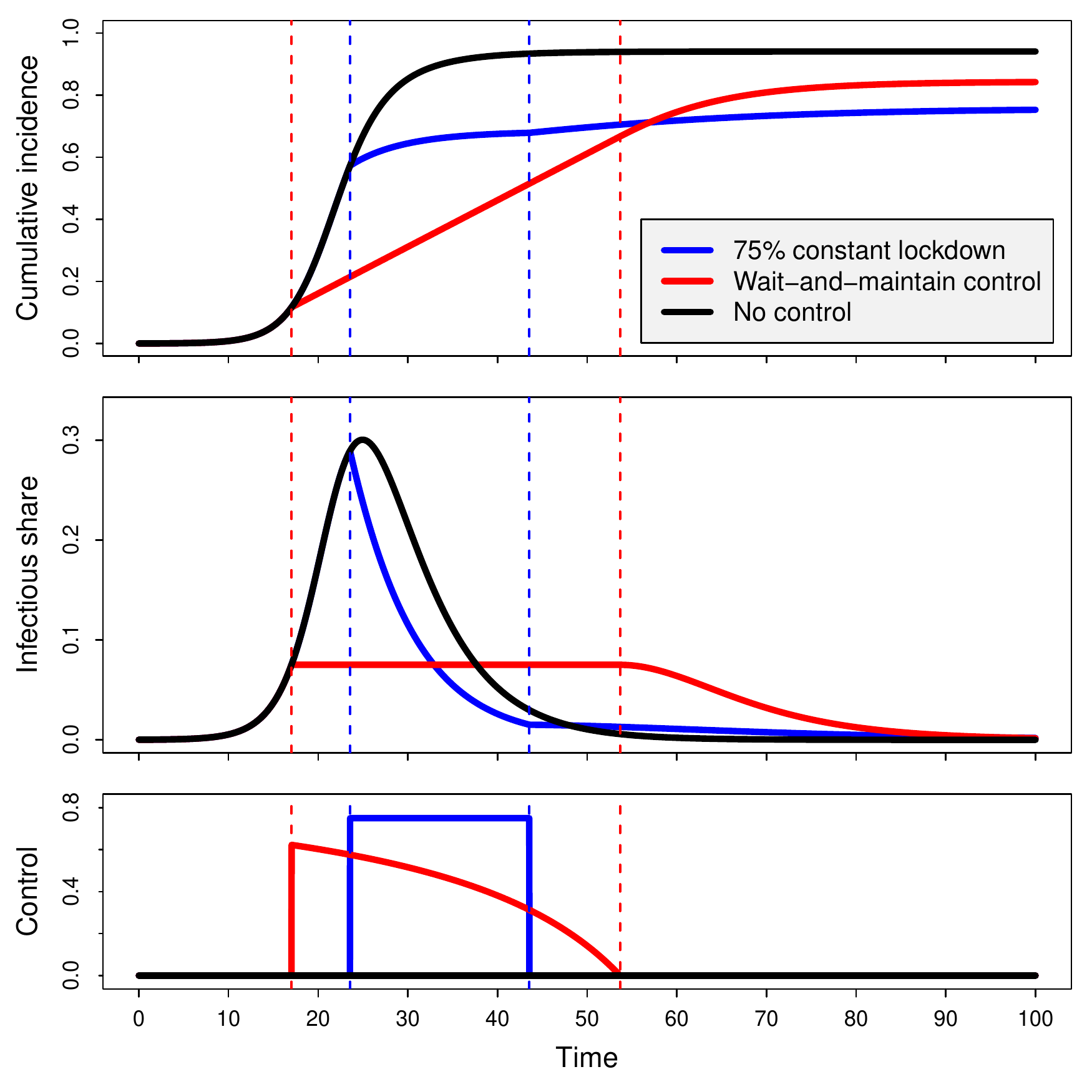}
\caption{Time plots of
cumulative incidence $1 - S(t)/S(0)$ (top),
prevalence $I(t)$ (middle), and intervention level $u(t)$ (bottom) for three strategies:
(a) Total-incidence-minimizing single lockdown of 20 days at 75\% level with $\onenorm{u}=15$, and optimized start time = 23.6 (blue).
(b) Peak-prevalence-minimizing
strategy \eqref{eq:WaitMaintainRelax} 
with start time 17.0 and duration 36.6 corresponding to $\onenorm{u}=15$ (red).
(c) No control (black).
The total incidence values for the three strategies are 0.758, 0.843, and 0.940. The corresponding peak prevalence values are 0.289, 0.075, and 0.300.
}
\label{fig:PerformanceVsMSW}
\end{figure}

\mysection{Discussion}

It was proven that, among all intervention strategies with cumulative cost not exceeding $c_1$ and magnitude never exceeding $\cmax$,
a single lockdown at level $\cmax$, duration $c_1/\cmax$ started at an optimal time instant minimizes the total number of individuals eventually getting infected. As a consequence, assuming that a fixed proportion of infected individuals end up in hospitals and another fixed proportion will die, this strategy also minimizes the cumulative number of hospitalizations and case fatalities. We also saw that the total incidence is at least $1-1/(S(0) \Rnaught)$ for any intervention with a finite total cost. This reflects the fact that in the absence of future vaccinations, the prevalence would start rising again if the susceptible share would be above the herd immunity level $1-1/\Rnaught$ at a time when interventions are relaxed.

Numerically it was seen that milder restrictions for a longer time had very little effect on the total incidence, and also, somewhat surprisingly, that restrictions imposed very early are not as effective as waiting until the prevalence has grown substantially before inserting maximal prevention. Rather counterintuitively, adding restrictions prior to the beginning of an existing intervention strategy may even result in a \emph{larger} total incidence. On the contrary, as a by-product of our main proof, it was proven that adding restrictions \emph{after} an intervention strategy has ended can only reduce the total incidence.

The optimization problem is formulated without having any fixed time horizon in mind. If for example a vaccine was known to become available not too far into the future, then this could lead to a rather different optimal prevention. Another assumption was that we were only willing to spend a finite cost $c_1$ for the cumulative preventions. For a severe disease this might not be the case --- we might be willing to keep some restrictive level for a very long time. In such a case the optimal solution will also be different.

A common feature of epidemic models based on deterministic differential equations is that the fraction of infected individuals never exactly reaches zero. In real epidemics, this might happen due to stochastic finite-population effects. In such cases, a feasible intervention strategy might be to aim for an early elimination of the disease by imposing massive restrictions early on, as has been advocated with SARS-CoV-2 in certain countries. The performance of such strategies cannot be analysed using the type of deterministic models considered here.

The class of interventions considered was assumed to have a bounded linear cost integrated over time. It is not obvious that the societal cost of preventive measures act linearly on the amount of prevention, and therefore nonlinear cost functionals might be relevant to consider. Similarly, changing an existing intervention level might incur additional societal costs, especially if such changes are carried out frequently. The analysis of optimal intervention strategies under such cost functions remains an open problem worthy of attention.

Other extensions worth considering would be to make the underlying epidemic model more realistic by incorporating seasonal effects, adding a latency period between which individuals have been infected but are not yet infectious, including a delay for the time it takes for an intervention decision to take effect, acknowledging different types of individuals and social structures such as households and workplaces, and relaxing the assumption that individuals recover at a constant rate. However, we do not expect major qualitative changes from these extensions, as opposed to considering nonlinear cost functions or unlimited intervention budgets.

\mysection{Proofs}

\mysubsection{Proof of Theorem~\ref{the:Main}}
\label{sec:Main}

Denote by $U(c_1, \cmax)$ the set of piecewise continuous functions $u:[0,\infty) \to [0,1]$
such that $\onenorm{u} \le c_1$ and $\supnorm{u} \le \cmax$.  
Denote by $J(u)$ the total incidence corresponding to intervention strategy $u \in U(c_1, \cmax)$.
Denote
\[
 J_* \weq \inf_{u \in U(c_1, \cmax)} J(u).
\]
Intuition suggests that interventions carried out in a distant future should have a negligible effect on the evolution of the epidemic.  Proposition~\ref{the:UI} in the appendix confirms this and tells that it is possible to select constants $C, \alpha, T_* > 0$ such that 
\begin{equation}
 \label{eq:Tail.I}
 \int_T^\infty I_u(t) \, dt
 \wle C e^{- \alpha T}
\end{equation}
for all $T \ge T_*$ and all controls bounded by $\onenorm{u} \le c_1$.
Let us now fix $\epsilon > 0$ and choose $u_1 \in U(c_1, \cmax)$ such that
\begin{equation}
 \label{eq:Ju0}
 J(u_1) \le J_* + \epsilon.
\end{equation}

(i) \emph{Truncation}. We will approximate $u_1$ by a control $u_2 = 1_{[0,T]} u_1$ of finite duration, where we choose a large enough $T > T_*$ so that $C e^{-\alpha T} \le \epsilon$. Because the trajectories $(S_{u_1}, I_{u_1})$ and $(S_{u_2}, I_{u_2})$ coincide up to time $T$, we see that
\[
 | J(u_2) - J(u_1) |
 \weq \frac{\gamma}{S(0)}  \left| \int_T^\infty I_{u_2}(t) \, dt - \int_T^\infty I_{u_1}(t) \, dt \right|
\]
Because $\onenorm{u_2 } \le \onenorm{u_1} \le c_1$, inequality \eqref{eq:Tail.I} tells that both integrals on the right are at most $C e^{-\alpha T}$, and it follows that
\begin{equation}
 \label{eq:Ju1}
 | J(u_2) - J(u_1) |
 \wle 2 \epsilon \frac{\gamma}{S(0)}.
\end{equation}

(ii) \emph{Quantization}. We will next approximate $u_2$ by a bang--bang control $u_3$ defined as follows.  We divide time into small intervals $I_k = ( (k-1)h, k h]$ of length $0 < h \le \epsilon e^{-(\beta+\gamma) T}$.  We quantize the control function $u_2$ in a frequency modulated fashion so that in each time $I_k$, first $u_3=0$ for $h-\tau_k$ time units, and then $u_3 = \cmax$ for the remaining $\tau_k$ time units, where 
\[
 \tau_k
  \weq b^{-1} \int_{h k - h}^{h k} u_2(t) \, dt
\]
is selected so that $\int_{I_k} u_3(t) \, dt = \int_{I_k} u_2(t) \, dt$, see Figure~\ref{fig:Quantization}. As a consequence of $h \le \epsilon e^{-(\beta+\gamma) T}$, it follows (details in Proposition~\ref{the:QuantizationFinal}) that
\begin{align*}
 \int_0^T \abs{I_{u_3}(t)-I_{u_2}(t)} \, dt
 &\wle \int_0^T 3 \beta h e^{(\beta+\gamma) t} \, dt \\
 &\wle 3 \beta h (\beta+\gamma)^{-1} e^{(\beta+\gamma) T}
 \wle 3 \epsilon.
\end{align*}
Furthermore,
\[
 \int_T^\infty \abs{I_{u_3}(t)-I_{u_2}(t)} \, dt
 \wle \int_T^\infty I_{u_2}(t) \, dt + \int_T^\infty I_{u_3}(t) \, dt,
\]
and by noting that $\onenorm{u_3} = \onenorm{u_2} \le c_1$ (Lemma~\ref{the:SlottingBangBang}),  inequality \eqref{eq:Tail.I} again guarantees that both integrals on the right are at most 
$C e^{-\alpha T} \le \epsilon$. Hence it follows that
\begin{equation}
 \label{eq:Ju2}
 | J(u_3) - J(u_2) |
 \wle 5 \epsilon \frac{\gamma}{S(0)}.
\end{equation}

\begin{figure}[h]
\centering
\includegraphics[width=80mm]{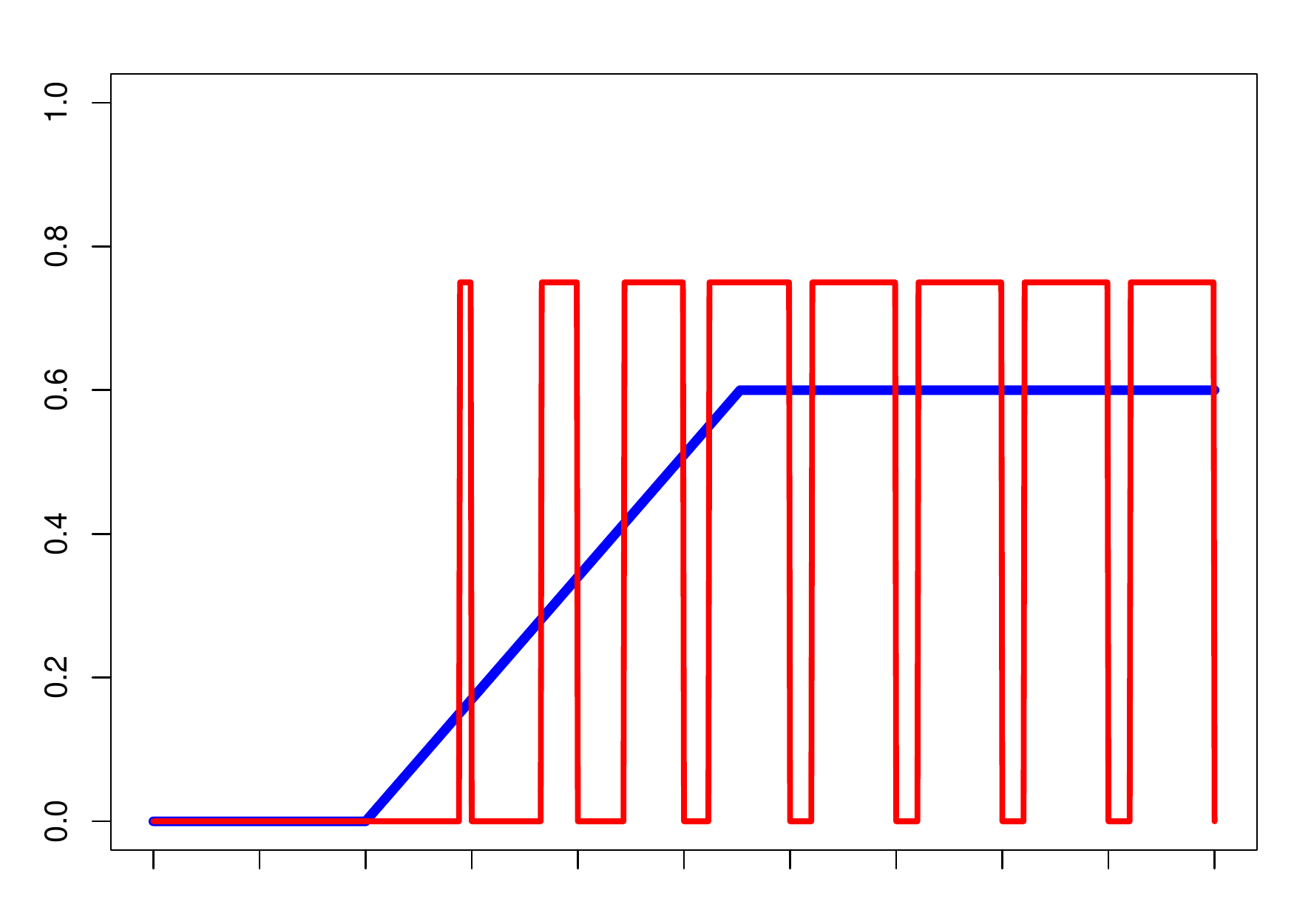}
\caption{\label{fig:Quantization} Quantization of a function $u$ (blue) by frequency modulated function $\hu$ with amplitude 0.75 (red).}
\end{figure}

(iii) \emph{Prolongation}. By construction, $u_3 = \cmax 1_A$ where $A$ is a finite union of intervals of total length $\abs{A}$. Because $\onenorm{u_3} = \cmax \abs{A}$ and $\onenorm{u_3} = \onenorm{u_2} \le c_1$, we find that $\abs{A} \le c_1/\cmax$.  Intuition suggests that a smaller disease burden will be incurred by replacing $A$ by a larger subset of the time axis, but such monotonicity properties are known to fail in general (see for example \cite{Chikina_Pegden_2020,Handel_Longini_Antia_2007}). Nevertheless, prolonging the {last} interval of $A$ from the end does increase the final susceptible share (details in Lemma~\ref{the:Monotonicity}).  Let us define $u_4 = \cmax 1_{\tilde A}$ where $\tilde A$ is obtained by prolonging the last interval in $A$ from the end so that $\abs{\tilde A} = c_1/\cmax$. Then it follows that $S_{u_4}(\infty) \ge S_{u_3}(\infty)$, and therefore, by recalling the equation $ S(0) + I(0) - S_u(\infty) \weq \gamma \onenorm{I_u}
$,
we find that
\begin{equation}
 \label{eq:Ju3}
 J(u_4) \wle J(u_3).
\end{equation}

(iv) \emph{Merging}. The control $u_4$ belongs to the set $\tilde U$ of all controls of the form $u = \cmax 1_B$ where $B \subset [0,\infty)$ is a union of finitely many disjoint intervals of total length $c_1/\cmax$. Among those, the best for minimizing cumulative incidence are those where all disjoint intervals are merged into a single interval of length $c_1/\cmax$. Especially, it is known \cite[Theorem 1.1]{Feng_Iyer_Li_2021} that there exists $\sigma \ge 0$ such that the control $u_5 = \cmax 1_{(\sigma, \sigma + c_1/\cmax)}$ satisfies $J(u_5) \weq \inf_{u \in \tilde U} J(u)$.  Especially, 
\begin{equation}
 \label{eq:Ju4}
 J(u_5) \wle J(u_4).
\end{equation}
By collecting the inequalities \eqref{eq:Ju0}--\eqref{eq:Ju4} together, we conclude that
\[
 J(u_5) \wle J_* + 8 \epsilon \frac{\gamma}{S(0)}.
\]
To summarize, for any $\epsilon > 0$ there exists $\sigma \ge 0$ such that the single-lockdown control $u_5 = \cmax 1_{(\sigma, \sigma + c_1/\cmax)}$ satisfies the above inequality. We conclude that
\begin{equation}
 \label{eq:Infimum1LD}
 \inf_{u \in U(c_1, \cmax)} J(u)
 \weq \inf_{\sigma \ge 0} J(v_\sigma),
\end{equation}
where $v_\sigma = \cmax 1_{(\sigma, \sigma + c_1/\cmax)}$ denotes a single-lockdown control with start time $\sigma$, duration $c_1/\cmax$, and constant intervention level $\cmax$.

(v) \emph{Timing}. Finally, we will verify that the infimum in \eqref{eq:Infimum1LD} is attained by an optimally chosen start time $\sigma_* \ge 0$.  Denote $\hat U = \cup_{\sigma \ge 0} \hat U_\sigma$ where $\hat U_\sigma$ is the set of controls such that $\supnorm{u} \le \cmax$ and $u=0$ outside the time interval $[\sigma, \sigma+c_1/\cmax]$. Then \cite[Theorem 1]{Bliman_Duprez_2021} implies
that there exists a unique $\sigma_* \ge 0$ such that
\[
 \sup_{u \in \hat U} S_u(\infty)
 \weq \sup_{\sigma \ge 0} \sup_{u \in \hat U_\sigma} S_u(\infty)
 \weq S_{v_{\sigma_*}}(\infty)
\]
for $v_{\sigma_*} = \cmax 1_{(\sigma, \sigma + c_1/\cmax)}$. Therefore,
by recalling the equation $ S(0) + I(0) - S_u(\infty) \weq \gamma \onenorm{I_u}.
$, 
it follows that 
\[
 \inf_{u \in \hat U} J(u)
 \weq J(v_{\sigma_*}).
\]
Because $\hat U$ contains all controls of the form $v_\sigma = \cmax 1_{(\sigma, \sigma + c_1/\cmax)}$, we conclude that
\[
 J(v_{\sigma_*})
 \weq \inf_{\sigma \ge 0} J(v_{\sigma}).
\]
In light of \eqref{eq:Infimum1LD}, this confirms the statement of the theorem.
\qed

\mysubsection{Proof of Theorem~\ref{the:Bounds}}

For a standard SIR model with no interventions ($u=0$), it is well known that the limiting susceptible share is bounded by $S(\infty) \le \fgb$. By a careful reasoning (Proposition~\ref{the:Time2HerdImmunity} in the appendix), the same bound extends to epidemics controlled by an intervention with $\onenorm{u} < \infty$. Therefore, the total incidence is bounded from below by
$1-S(\infty)/S(0) \ge 1-\gamma/(\beta S(0))$.

For the upper bound, it appears intuitively clear that the limiting susceptible share $S_u(\infty)$ in an epidemic with an arbitrary control $u$ is larger than equal than the corresponding quantity $S_0(\infty)$ in an epidemic with no interventions. Indeed, Chikina and Pegden \cite{Chikina_Pegden_2020} have shown that a pointwise ordering of interventions $u_1 \le u_2$ implies $S_{u_1}(\infty) \le S_{u_2}(\infty)$ under an extra assumption that $u_2$ is nondecreasing. Straightforward modifications of the analysis in \cite{Chikina_Pegden_2020} show that the above implication holds also when $u_1$ instead of $u_2$ is required to be nondecreasing. By selecting $u_1=0$ and letting $u_2=u$ be an arbitrary piecewise continuous intervention, we find that
$S_0(\infty) \le S_u(\infty)$, and we conclude that the total incidence is bounded from above by
$1-S_u(\infty)/S(0)
\le 1-S_0(\infty)/S(0)$.
\qed

\ifarxiv
\paragraph{Acknowledgments}
The authors are grateful for financial support from project 105572 NordicMathCovid as part of the Nordic Programme on Health and Welfare funded by NordForsk.
\fi
\ifpnas
\acknow{The authors are grateful for financial support from project 105572 NordicMathCovid as part of the Nordic Programme on Health and Welfare funded by NordForsk.}
\showacknow{}
\fi

 \ifarxiv
\bibliographystyle{unsrt}
\fi
\bibliography{lslReferences}

\newcommand{\SortNoop}[1]{}\def\cprime{$'$}
\begin{thebibliography}{10}

\bibitem{Kermack_McKendrick_1927}
William~Ogilvy Kermack and A.~G. McKendrick.
\newblock A contribution to the mathematical theory of epidemics.
\newblock {\em Proceedings of the Royal Society of London A},
  115(772):700--721, 1927.

\bibitem{Diekmann_Heesterbeek_Britton_2013}
Odo Diekmann, Hans Heesterbeek, and Tom Britton.
\newblock {\em Mathematical Tools for Understanding Infectious Disease
  Dynamics}.
\newblock Princeton University Press, 2013.

\bibitem{Morris_Rossine_Plotkin_Levin_2021}
Dylan~H. Morris, Fernando~W. Rossine, Joshua~B. Plotkin, and Simon~A. Levin.
\newblock Optimal, near-optimal, and robust epidemic control.
\newblock {\em Communications Physics}, 4(1):78, 2021.

\bibitem{Greene_Sontag_2021}
James~M. Greene and Eduardo~D. Sontag.
\newblock Minimizing the infected peak utilizing a single lockdown: a technical
  result regarding equal peaks.
\newblock medRxiv, 2021.

\bibitem{Miclo_Spiro_Weibull_2020}
Laurent Miclo, Daniel Spiro, and J{\"o}rgen Weibull.
\newblock Optimal epidemic suppression under an {ICU} constraint.
\newblock arXiv, 2020.

\bibitem{Avram_Freddi_Goreac_2022}
Florin Avram, Lorenzo Freddi, and Dan Goreac.
\newblock Optimal control of a {SIR} epidemic with {ICU} constraints and target
  objectives.
\newblock {\em Applied Mathematics and Computation}, 418:126816, 2022.

\bibitem{Feng_Iyer_Li_2021}
Yuanyuan Feng, Gautam Iyer, and Lei Li.
\newblock Scheduling fixed length quarantines to minimize the total number of
  fatalities during an epidemic.
\newblock {\em Journal of Mathematical Biology}, 82(7):69, 2021.

\bibitem{Bliman_Duprez_Privat_Vauchelet_2021}
Pierre-Alexandre Bliman, Michel Duprez, Yannick Privat, and Nicolas Vauchelet.
\newblock Optimal immunity control and final size minimization by social
  distancing for the {SIR} epidemic model.
\newblock {\em Journal of Optimization Theory and Applications},
  189(2):408--436, 2021.

\bibitem{Ketcheson_2021}
David~I. Ketcheson.
\newblock Optimal control of an sir epidemic through finite-time
  non-pharmaceutical intervention.
\newblock {\em Journal of Mathematical Biology}, 83(1), Jun 2021.

\bibitem{Bliman_Duprez_2021}
Pierre-Alexandre Bliman and Michel Duprez.
\newblock How best can finite-time social distancing reduce epidemic final
  size?
\newblock {\em Journal of Theoretical Biology}, 511:110557, 2021.

\bibitem{Cianfanelli_etal_2022}
Leonardo Cianfanelli, Francesca Parise, Daron Acemoglu, Giacomo Como, and
  Asuman Ozdaglar.
\newblock Lockdown interventions in {SIR} model: {Is} the reproduction number
  the right control variable?, 2021.
\newblock arXiv.

\bibitem{Flaxman_etal_2020}
Seth Flaxman, Swapnil Mishra, Axel Gandy, H.~Juliette~T. Unwin, Thomas~A.
  Mellan, Helen Coupland, Charles Whittaker, Harrison Zhu, Tresnia Berah,
  Jeffrey~W. Eaton, M{\'e}lodie Monod, Pablo~N. Perez-Guzman, Nora Schmit,
  Lucia Cilloni, Kylie E.~C. Ainslie, Marc Baguelin, Adhiratha Boonyasiri,
  Olivia Boyd, Lorenzo Cattarino, Laura~V. Cooper, Zulma Cucunub{\'a}, Gina
  Cuomo-Dannenburg, Amy Dighe, Bimandra Djaafara, Ilaria Dorigatti, Sabine~L.
  van Elsland, Richard~G. FitzJohn, Katy A.~M. Gaythorpe, Lily Geidelberg,
  Nicholas~C. Grassly, William~D. Green, Timothy Hallett, Arran Hamlet, Wes
  Hinsley, Ben Jeffrey, Edward Knock, Daniel~J. Laydon, Gemma Nedjati-Gilani,
  Pierre Nouvellet, Kris~V. Parag, Igor Siveroni, Hayley~A. Thompson, Robert
  Verity, Erik Volz, Caroline~E. Walters, Haowei Wang, Yuanrong Wang, Oliver~J.
  Watson, Peter Winskill, Xiaoyue Xi, Patrick G.~T. Walker, Azra~C. Ghani,
  Christl~A. Donnelly, Steven Riley, Michaela A.~C. Vollmer, Neil~M. Ferguson,
  Lucy~C. Okell, Samir Bhatt, and Imperial College COVID-19~Response Team.
\newblock Estimating the effects of non-pharmaceutical interventions on
  {COVID-19} in {Europe}.
\newblock {\em Nature}, 584(7820):257--261, 2020.

\bibitem{DiLauro_Kiss_Miller_2021}
Francesco Di~Lauro, Istv{\'a}n~Z. Kiss, and Joel~C. Miller.
\newblock Optimal timing of one-shot interventions for epidemic control.
\newblock {\em PLOS Computational Biology}, 17(3):1--25, 03 2021.

\bibitem{Chikina_Pegden_2020}
Maria Chikina and Wesley Pegden.
\newblock Failure of monotonicity in epidemic models, 2020.

\bibitem{Handel_Longini_Antia_2007}
Andreas Handel, Ira~M Longini, and Rustom Antia.
\newblock What is the best control strategy for multiple infectious disease
  outbreaks?
\newblock {\em Proceedings of the Royal Society B: Biological Sciences},
  274(1611):833--837, 2007.

\bibitem{Kruse_Strack_2020}
Thomas Kruse and Philipp Strack.
\newblock Optimal control of an epidemic through social distancing.
\newblock SSRN Preprint, 2020.

\bibitem{Khalil_1995}
Hassan~K. Khalil.
\newblock {\em Nonlinear Systems}.
\newblock Prentice Hall, 1995.

\end{thebibliography}

\clearpage 
\appendix

\section{Truncation}
\label{sec:Truncation}

The appendices present technical details needed for proving Theorem~\ref{the:Main}.
Appendix~\ref{sec:Truncation} contains details related to approximating a control of infinite time horizon by truncation.
Appendix~\ref{sec:Quantization} contains an analysis of a frequency-modulated quantization operator.
Appendix~\ref{sec:Prolongation} contains a monotonicity property related to minimizing the disease burden by prolonged interventions.

\subsection{Time to reach herd immunity}

For an epidemic trajectory $(S_u, I_u)$ controlled by $u$, we denote the time at which herd immunity is reached by
\[
 t_H(u)
 \weq \inf\left\{t \ge 0: S_u(t) \le \fgb \right\}.
\]
The following result generalizes \cite[Lemma 3.3:(2)]{Feng_Iyer_Li_2021} (and corrects a minor mistake in its proof).

\begin{proposition}
\label{the:Time2HerdImmunity}
For any initial state with $S(0) > \fgb$ and $I(0)>0$, and any piecewise continuous
control such that $\norm{u}_1 < \infty$, the time to reach herd immunity is finite and bounded by
\[
 t_H(u)
 \wle \norm{u}_1 + \frac{\log(\fbg S(0))}{\beta I(0)} e^{\gamma \norm{u}_1}.
\]
\end{proposition}
\begin{proof}
We will analyse the system on the time interval up to $t_H = t_H(u)$. Observe that the logarithmic state variables evolve according to
\begin{align}
 \label{eq:LogS}
 (\log S_u)' &\weq -\beta (1-u) I_u, \\
 \label{eq:LogI}
 (\log I_u)' &\weq \beta (1-u) S_u - \gamma.
\end{align}
By definition, $S_u \ge \fgb$ on $[0,t_H]$. This lower bound combined with \eqref{eq:LogI} implies that $(\log I_u)' \ge -\gamma u$, and therefore
\begin{equation}
 \label{eq:iLower}
 I_u(t)
 \wge I(0) e^{-\gamma \int_0^t u(s) ds}
 \wge I(0) e^{-\gamma \onenorm{u}}
\end{equation}
on $[0,t_H]$. By denoting $I_{\rm min} = I(0) e^{-\gamma \onenorm{u}}$ and combining the above inequality with \eqref{eq:LogS}, it follows that $(\log S_u)' \le -\beta I_{\rm min} (1-u)$ on $[0,t_H]$, so that
\begin{align*}
 \log S_u(t_H) - \log S(0)
 &\wle - \beta I_{\rm min} \int_0^{t_H} (1-u(t)) \, dt \\
 &\wle \beta I_{\rm min} ( \onenorm{u} - t_H).
\end{align*}
By noting that $S_u(t_H) = \fgb$ by the continuity of $S_u$, we see that
\[
 t_H
 \wle \onenorm{u} + \frac{\log S(0) - \log \fgb}{\beta I_{\rm min}},
\]
and the claim follows.
\end{proof}

\subsection{Uniform integrability}

The following result shows that the collection of infectious trajectories induced by controls bounded by $\norm{u}_1 \le c_1$ is uniformly integrable.

\begin{proposition}
\label{the:UI}
For any $\beta,\gamma>0$, any initial state with $S(0) > 0$ and $I(0)>0$, and 
any $c_1 \ge 0$, there exist constants $\alpha, C, T_* > 0$ such that
\[
 \sup_{ \onenorm{u} \le c_1} \int_T^\infty I_u(t) \, dt
 \wle C e^{- \alpha T}
 \qquad \text{for all $T \ge T_*$}.
\]
\end{proposition}
\begin{proof}
(i)  Proposition~\ref{the:Time2HerdImmunity} implies that for all controls with $\onenorm{u} \le c_1$, the time to reach herd immunity is bounded by the constant
\[
 t^*_H \weq \max\{ c_1 + \beta^{-1} I(0)^{-1} \log(\fbg S(0)) e^{\gamma c_1}, \, 0\}.
\]
Therefore, the susceptible share satisfies $S_u(t) \le \fgb$ from time $t^*_H$ onwards. We will choose a slightly larger time horizon $T_* = t^*_H + c_1 + 1$ and show that
\begin{equation}
 \label{eq:SAfterHerdImmunity}
 S_u(t) \wle (1-\delta) \fgb
\end{equation}
for all $t \ge T_*$ and all controls $u$ such that $\onenorm{u} \le c_1$, where $\delta = 1 - \exp(- \beta I(0) e^{- \gamma T_*})$. To verify \eqref{eq:SAfterHerdImmunity}, observe first that
\begin{equation}
 \label{eq:ControlAfterHerdImmunity}
 \int_{t_H^*}^{T_*} (1-u(t)) \, dt 
 \weq c_1 + 1 - \int_{t_H^*}^{T_*} u(t) \, dt 
 \wge c_1 + 1 - \onenorm{u}
 \wge 1.
\end{equation}
The crude lower bound $(\log I_u)' = \beta (1-u) S_u - \gamma \ge -\gamma$ implies that $I(t) \ge I(0) e^{-\gamma T_*}$ on $(0, T_*)$. By noting that $(\log S)' = - \beta (1-u) I$, it follows by \eqref{eq:ControlAfterHerdImmunity} that
\begin{align*}
 \log S_u(T_*) - \log S_u(t_H^*)
 \weq - \beta \int_{t_H^*}^{T_*} (1-u(t)) I_u(t) \, dt
 &\wle - \beta I(0) e^{-\gamma T_*}.
\end{align*}
By noting that the right side above equals $\log(1-\delta)$, and that
$S_u(t_H^*) \le \fgb$ due to $t_H(u) \le t_H^*$, we conclude that
$
 \log S_u(T_*)
 \le \log \fgb + \log(1-\delta),
$
and that \eqref{eq:SAfterHerdImmunity} is valid.

(ii) We will next derive an upper bound for the tail integrals of $I_u$.   By applying \eqref{eq:SAfterHerdImmunity}, we see that for all $t > T_*$,
\[
 (\log I_u)'
 \weq \beta (1-u) S_u - \gamma
 \wle \beta (1-\delta) \fgb - \gamma
 \weq -  \gamma \delta,
\]
so that
\[
 I_u(t)
 \wle I_u(T_*) e^{-  \gamma \delta (t-T_*)}
 \wle e^{- \gamma \delta (t-T_*)}.
\]
By integrating the above inequality, we see that
\[
 \int_{T}^\infty I_u(t) \, dt
 \wle \frac{1}{ \gamma \delta} e^{-  \gamma \delta(T-T_*)}
 \qquad \text{for all $T \ge T_*$}.
\]
Therefore, the claim holds for $\alpha = \gamma \delta$ and $C = \frac{1}{ \gamma \delta}
e^{\gamma \delta T_*}$.
\end{proof}

\section{Quantization}
\label{sec:Quantization}

We develop a frequency modulation approach to approximate a general piecewise continuous function with a square waveform having a small wavelength. Denote by $U(b)$ the set of piecewise continuous functions $u: [0,\infty) \to [0,1]$ such that $\supnorm{u} \le b$. Given an amplitude $b > 0$ and wavelength $h > 0$, we define a quantization operator $Q_{b,h}: U(b) \to U(b)$ by setting $Q_{b,h} u = \hu$ with
\begin{equation}
 \label{eq:BangBang}
 \hu(t)
 \weq \begin{cases}
  b &\quad \text{if $t \in I$}, \\
  0 &\quad \text{otherwise},
 \end{cases}
\end{equation}
where $I = \cup_{k \ge 1} I_k$ is a union of disjoint intervals $I_k = [h k - \tau_k, h k)$ having lengths
\[
 \tau_k
 \weq b^{-1} \int_{h k - h}^{h k} u(t) \, dt.
\]
We will show that $\hat u$ approximates $u$ well in a weak sense for small wavelengths. We first prove a general approximation property (Lemma~\ref{the:SlottingBangBang}), and then derive an approximation result of an epidemic trajectory controlled by $\hu$ (Proposition~\ref{the:QuantizationFinal}).

A function $\phi: [0,\infty) \to \R$ is called locally bounded if 
$\supnormt{\phi} = \sup_{0 \le s \le t} \abs{\phi(s)}$ is finite for all $t$, and locally
Lipschitz continuous if 
$\Lipnormt{\phi} = \sup_{0 \le t_1 < t_2 \le t} \frac{\abs{\phi(t_2) - \phi(t_1)}}{t_2-t_1}$
is finite for all $t$.

\begin{lemma}
\label{the:SlottingBangBang}
For any $b, h > 0$ and $u \in U(b)$, 
the approximation $\hu = Q_{b,h} u$ defined by \eqref{eq:BangBang} satisfies 
$\onenorm{\hu} = \onenorm{u}$, and
\[
 \left| \int_0^t \big( \hu(s) - u(s) \big) \phi(s) \, ds \, \right|
 \wle b h \left( \supnormt{\phi} + t \Lipnormt{\phi} \right)
\]
for all $t \ge 0$ and all locally bounded and locally Lipschitz continuous $\phi$.
\end{lemma}
\begin{proof}
On each interval $[h k - h, h k)$, the quantized control alternates so that $\hu = 0$ for the first $h-\tau_k$ time units, and then $\hu = b$ for the remaining $\tau_k$ time units. The choice of $\tau_k$ then implies that
\begin{equation}
 \label{eq:IntegralInterval}
 \int_{h k - h}^{h k} \hu(s) \, ds
 \weq
 \int_{h k - h}^{h k} u(s) \, ds
\end{equation}
for all $k \ge 1$. By summing both sides of the above equality with respect to $k \ge 1$, we find that $\onenorm{\hu} = \onenorm{u}$.

Let us now fix $t \ge 0$. Denote $n = \floor{\frac{t}{h}}$ and $t_k = k h$, and observe that $\int_0^t \big( \hu(s) - u(s) \big) \phi(s) \, ds = \sum_{k=1}^n A_k + B_n$, where
$
 A_k
 = \int_{t_{k-1}}^{t_k} \big( \hu(s) - u(s) \big) \phi(s) \, ds
$
and 
$
 B_n
 = \int_{t_n}^t \big( \hu(s) - u(s) \big) \phi(s) \, ds.
$
By applying \eqref{eq:IntegralInterval}, we find that
\begin{align*}
 A_k
 \weq \int_{t_{k-1}}^{t_k} \big( \hu(s) - u(s) \big) \big( \phi(s) - \phi(t_{k-1}) \big) \, ds.
\end{align*}
Because $\abs{\phi(s) - \phi(t_{k-1})} \le h \Lipnormt{\phi}$ for all $t_{k-1} < s < t_k$, 
and $\supnorm{\hu - u} \le b$, it follows that
\begin{align*}
 | A_k |
 \wle h \Lipnormt{\phi} \int_{t_{k-1}}^{t_k} \big| \hu(s) - u(s) \big| \, ds
 \wle b h^2 \Lipnormt{\phi}.
\end{align*}
Furthermore, by noting that $t_n \le t < t_n + h$, it follows that
\[
 \abs{B_n}
 \wle \supnormt{\phi} \int_{t_n}^t \big| \hu(s) - u(s) \big| \, ds
 \wle b h \supnormt{\phi}.
\]
We conclude that
\begin{align*}
 \left| \int_0^t \big( \hu(s) - u(s) \big) \phi(s) \, ds \, \right|
 &\wle \sum_{k=1}^n \abs{A_k} + \abs{B_n} \\
 &\wle b n h^2 \Lipnormt{\phi} +b h \supnormt{\phi} \\
 &\wle b h t \Lipnormt{\phi} +  b h \supnormt{\phi}.
\end{align*}
\end{proof}

\begin{proposition}
\label{the:QuantizationFinal}
For any $b, h > 0$ and $u \in U(b)$, any $\beta,\gamma > 0$ and any initial state with $S(0),I(0) > 0$, the epidemic trajectories $(S_u, I_u)$ and $(S_{\hu}, I_{\hu})$ associated with $u$ and $\hu = Q_{b,h} u$ satisfy
\[
 \max\Big\{ \abs{S_{\hu}(t) - S_u(t)}, \, \abs{I_{\hu}(t)-I_u(t)} \Big\}
 \wle 3 \beta b h e^{(\beta+\gamma) t}
\]
for all $t \ge 0$.
\end{proposition}
\begin{proof}
The epidemic trajectory $(S_u,I_u)$ controlled by $u$ is the unique solution to $X'(t) = f(t,X(t))$, $X(0)=(S(0),I(0))$, where $f: [0,\infty) \times \R^2 \to \R^2$ is defined by
\begin{equation}
 \label{eq:SIRGenerator}
 f(t,X)
 \weq
 \begin{bmatrix}
  - \alpha(t) X_1 X_2, \\
  \alpha(t) X_1 X_2 - \gamma X_2
 \end{bmatrix}
\end{equation}
and $\alpha(t) = \beta(1-u(t))$.
The state space of the system is denoted by $\cX = \left\{X \in [0,1]^2: X_1 + X_2 \le 1 \right\}$ and we equip it with max norm $\norm{X} = \max\{\abs{X_1}, \abs{X_2}\}$. The Jacobian matrix of $X \mapsto f(t,X)$ equals
\[
 \partial_X f(t,X)
 \weq \begin{bmatrix}
  - \alpha(t) X_2 & -\alpha(t) X_1 \\
  \alpha(t) X_2 & \alpha(t) X_1 - \gamma
 \end{bmatrix}.
\]
The operator norm of the matrix $\partial_X f(t, X)$ induced by the max norm on $\R^2$ is the maximum absolute row sum, which is bounded by
\begin{align*}
 \norm{\partial_X f(t, X)}
 &\weq \max \Big\{ \alpha(t) ( \abs{X_1}+\abs{X_2} ), \,
 \alpha(t) \abs{X_2} + \abs{ \alpha(t) X_1 - \gamma } \Big\} \\
 &\wle \alpha(t) ( \abs{X_1}+\abs{X_2} ) + \gamma.
\end{align*}
Hence $\norm{\partial_X f(t,X)} \le \beta + \gamma$ for all $X \in \cX$ and $t \ge 0$, it follows \cite[Lemma 2.2]{Khalil_1995} that $f$ is Lipschitz continuous according to
\begin{equation}
 \label{eq:fLip}
 \norm{ f(t,X)-f(t,Y) }
 \le (\beta+\gamma) \norm{X-Y}
 \quad \text{for all $X,Y \in \cX$ and $t \ge 0$.}
\end{equation}

Let $Y: [0,\infty) \to \cX$ be the unique solution to
$Y'(t) = g(t,Y(t))$,
$Y(0) = (S(0),I(0))$,
where $g(t,X)$ is defined analogously to $f(t,X)$ but with $\alpha(t)$ in \eqref{eq:SIRGenerator} replaced by $\halpha(t) = \beta(1-\hu(t))$.
Then
\begin{align*}
 X(t) &\weq X(0) + \int_0^t f(s, X(s)) \, ds, \\
 Y(t) &\weq Y(0) + \int_0^t g(s, Y(s)) \, ds.
\end{align*}
Because $X(0)=Y(0)$, we find that the difference $Z(t) = Y(t) - X(t)$ satisfies
\begin{align*}
 Z(t)
 &\weq \int_0^t \Big( f(s, Y(s)) - f(s,X(s)) \Big) \, ds \\
 &\qquad + \int_0^t \Big( g(s, Y(s)) - f(s,Y(s)) \Big) \, ds.
\end{align*}
Inequality \eqref{eq:fLip} shows that
$\norm{f(s, Y(s)) - f(s,X(s))} \le (\beta+\gamma) \norm{z(s)}$ for all $s$, and we conclude that the max norm of $Z(t)$ is bounded by
\begin{equation}
 \label{eq:Gronwall1}
 \norm{Z(t)}
 \wle \int_0^t (\beta+\gamma) \norm{Z(s)} \, ds + h(t),
\end{equation}
where
\[
 h(t) \weq \left| \left| \, \int_0^t \Big( g(s, Y(s)) - f(s,Y(s)) \Big) \, ds \, \right| \right|.
\]

Observe next that
\[
 g(s, Y(s)) - f(s,Y(s))
 \weq \beta (u(s) - \hu(s)) \phi(s)
 \begin{bmatrix}
  - 1 \\
  + 1
 \end{bmatrix},
\]
where $\phi(s) = Y_1(s) Y_2(s)$, and hence
\[
 h(t)
 \weq \beta \left| \int_0^t (u(s) - \hu(s)) \phi(s) \, ds \right|.
\]
We also find that
$\abs{Y_1'(t)} = \abs{S_{\hu}(t)} \le \beta$
and 
$\abs{Y_2'(t)} = \abs{I_{\hu}(t)} \le \beta + \gamma$ for all $t \ge 0$.
Therefore, $\abs{\phi'(t)} \le 2 (\beta+\gamma)t$, and we conclude that $\phi$ is globally Lipschitz continuous according to $\abs{\phi(t)-\phi(s)} \le 2(\beta+\gamma) \abs{t-s}$ for all $s,t \ge 0$. Lemma~\ref{the:SlottingBangBang} then implies that 
\[
 h(t)
 \wle C_1 + C_2 t 
\]
for $C_1 = \beta b h$ and $C_2 = 2 \beta b h (\beta+\gamma)$. 
By combining this with \eqref{eq:Gronwall1}, we conclude that
\[
 \norm{Z(t)}
 \wle C_1 + C_2 t + \int_0^t (\beta+\gamma) \norm{Z(s)} \, ds,
\]
and Grönwall's inequality (e.g.\ \cite[Lemma 2.1]{Khalil_1995}) then implies that
\[
 \norm{Z(t)}
 \wle C_1 + C_2 t + \int_0^t (C_1 + C_2 s) (\beta+\gamma) e^{(\beta+\gamma)(t-s)} ds.
\]
Integration by parts shows that the right side equals
$C_1 e^{(\beta+\gamma) t} + \frac{C_2}{\beta+\gamma} \left( e^{(\beta+\gamma) t} - 1 \right)$,
from which we conclude that
\[
 \norm{Z(t)}
 \wle \left( C_1+ \frac{C_2}{\beta+\gamma} \right) e^{(\beta+\gamma) t}
 \weq 3 \beta b h e^{(\beta+\gamma) t},
\]
confirming the claim.
\end{proof}

\section{Prolongation}
\label{sec:Prolongation}

\subsection{A special function}

\begin{lemma}
\label{the:Special}
For any $\rho > 0$, the function
$f(x) = x - \rho^{-1} \log x$ is
strictly decreasing on $(0,\rho^{-1}]$, and
strictly increasing on $[\rho^{-1},\infty)$,
and has a unique minimum value $y_0 = \frac{1}{\rho}(1+\log \rho)$.
The restriction of $f$ into $(0,\rho^{-1}]$ is invertible, and the corresponding inverse function
$g: [y_0, \infty) \to (0,\rho^{-1}]$ is strictly decreasing.
\end{lemma}
\begin{proof}
Because $f'(x) = 1 - \rho^{-1} x^{-1}$ we find that $f'(x) < 0$ for $s < \rho^{-1}$, and $f'(x) > 0$ for $x > \rho^{-1}$. Hence the stated monotonicity properties follow. It also follows that $f$ restricted to $(0,\rho^{-1}]$ has a well-defined inverse function $g$. The inverse function rule $g'(y) = \frac{1}{f'(g(y))}$ shows that $g'(y) < 0$ for all $y > y_0$, and confirms that $g$ is strictly decreasing on $[y_0,\infty)$.
\end{proof}

\subsection{Prolonged interventions imply less infections}

\begin{lemma}
\label{the:Monotonicity}
Let $(S_{1},I_{2})$ be an epidemic trajectory
controlled by $u_1$ such that $u_1=0$ outside $[0,T]$. Let $(S_{2},I_{2})$ be an epidemic trajectory with the same initial state but a modified control $u_2 = u_1 + c 1_{[t_1,t_2]}$ with $T \le t_1 \le t_2$. Then $S_1(\infty) \le S_2(\infty)$.
\end{lemma}

\begin{proof}
The epidemic under a control $u_k$ evolves according to $S_k' = -\beta (1-u_k) S_k I_k$ and $I_k' = \beta (1-u_k) S_k I_k - \gamma I_k$. This implies that $(S_k+I_k)' = - \gamma I_k$ and $(\log S_k)' = - \beta (1-u_k) I_k$. As a consequence, we find that 
\begin{equation}
 \label{eq:ControlInvariant}
 \Big( \beta (1-u_k) (S_k+I_k) - \gamma \log S_k \Big)'
 \weq 0
\end{equation}
at every time instant in which $u_k' = 0$.

Let $V_k = S_k + I_k - \fgb \log S_k$, the \emph{vulnerability} of the population under control $u_k$. Because both controls vanish on $(t_2,\infty)$ and $I_1(\infty) = I_2(\infty) = 0$, we find from \eqref{eq:ControlInvariant} that
\begin{equation}
 \label{eq:FinalState}
 S_k(\infty) - \fgb \log S_k(\infty)
 \weq V_k(t_2), \quad k=1,2.
\end{equation}
Because both trajectories are equal up to time $t_1$, we see that $V_1(t_1) = V_2(t_1)$.
Because $u_1=0$ on the interval $(t_1, t_2)$, \eqref{eq:ControlInvariant} shows that $V_1(t_1) = V_1(t_2)$.  Analogously, by noting that $u_2=c$ on $(t_1, t_2)$, it follows that  
\begin{align*}
 & (1-c) (S_2(t_1)+I_2(t_1)) - \fgb \log S_2(t_1) \\
 &\weq (1-c) (S_2(t_2)+I_2(t_2)) - \fgb \log S_2(t_2),
\end{align*}
which can be rewritten as
\[
 V_2(t_2) 
 \weq V_2(t_1) + c (S_2(t_2)+I_2(t_2) - S_2(t_1) - I_2(t_1)).
\]
Because $(S_2+I_2)' = - \gamma I_2$, we see that $S_2(t)+I_2(t)$ is decreasing, and therefore,
\[
 V_2(t_2) 
 \wle V_2(t_1)
 \weq V_1(t_1)
 \weq V_2(t_1).
\]
Because of \eqref{eq:FinalState}, it follows (Lemma~\ref{the:Special}) that $S_1(\infty) \le S_2(\infty)$.
\end{proof}

\end{document}